\documentclass[a4paper]{amsart}
\usepackage{amsmath, amsfonts, amsthm}
\usepackage{amssymb,graphics,epsfig}
\usepackage{enumerate}

\newtheorem{defn}{Definition}
\newtheorem{thm}{Theorem}

\newcommand{\Lun}{L^1}
\newcommand{\Lunloc}{L^1_{loc}}

\newcommand{\Ldeloc}{L^2_{loc}}

\newcommand{\Linfloc}{L^{\infty}_{loc}}
\newcommand{\tr}{\chi_{E<m}}

\newcommand{\phieps}{\phi_{\epsilon}}

\newcommand{\RR}{\mathbb{R}}

\newcommand{\NN}{\mathbb{N}}

\newcommand{\Com}{C^\infty_0}
\newcommand{\CCC}{C}
\newcommand{\Linf}{L^\infty}

\makeatletter \@addtoreset{equation}{section} \makeatother

\newcommand{\cqfd}{{\unskip\kern 6pt\penalty 500
\raise -2pt\hbox{\vrule\vbox to 6pt{\hrule width 6pt
\vfill\hrule}\vrule}\par}}

\title{On Liouville transport equation with a force field in $BV_{loc}$}

\author{Maxime Hauray}

\address{M. Hauray: Universit\'e Paris-Dauphine, et CNRS}

\email{maxime.hauray@univ-amu.fr}

\begin{document}



\begin{abstract} We prove the existence and uniqueness of renormalized 
solutions of the Liouville equation for $n$ particles with a interaction 
potential in $BV_{loc}$ execpt at the origin. This implies the existence and 
uniqueness of a a.e. flow solution of the associated ODE.
\end{abstract}

\maketitle

\section{Introduction}
We consider the Liouville (or transport) equation
\begin{equation} \label{liou_nc}
\frac{\partial f}{\partial t} + \sum_{i=1}^{n} v_i \cdot \nabla_{x_i}
  f - \sum_{i \neq j} \nabla V(x_i -x_j) \cdot \nabla_{v_i}
  f = 0
\end{equation}
with initial conditions
\begin{equation} \label{init_nc}
f(0,x_1,\dots,x_n,v_1,\dots,v_n)=f^0(x_1,\dots,x_n,v_1,\dots,v_n)
\end{equation}
Here $n \in \NN$, each $x_i$ and $v_i$ belongs to $\RR^d$ for some $d \geq 1$, $f$ is a real function defined on $[0,\infty) \times \RR^{2dn}$. We shall  always assume that the interaction potential $V$ is such that $\nabla V \in BV_{loc}(\RR^d - 0)$. Our goal here is to show the existence and the uniqueness of solutions (in a sense to be made more precise) of (\ref{liou_nc})-(\ref{init_nc}) for each $f^0 \in \Lunloc(\RR^{2dn})$.

As is well know, this equation is in some sense equivalent to the system of ODE
\begin{equation} \label{ODEnpart}
\left\{ \begin{array}{l} 
\dot{X}_i(t)=V_i \\
\dot{V}_i(t)= - \sum_{i\neq j} \nabla V(X_i(t) -X_j(t))
\end{array} \right. \qquad \forall \; 1 \leq i \leq n
\end{equation}
We shall also prove the existence and the uniqueness of a flow
solution of (\ref{ODEnpart}) (in a sense to be made more precise).

Here, we will use the method of resolution of transport equations and
associated ODE introduced by R. DiPerna and P.L. Lions in 1989 in \cite{DPL}. In
this paper, they prove the existence and the uniqueness of the solution of 
a transport equation 
when the vector field belongs to $W^{1,1}_{loc}$, and use this to obtain a 
unique flow solution of the ODE. In the note \cite{Lio}, 
P.L. Lions 
extend this result to piecewise $W^{1,1}$ vector-field and give a clearer
proof of the equivalence between the existence and uniqueness of a
solution of the transport equation and the existence and the
uniqueness of a flow solution of the ODE. In \cite{Bou}, F. Bouchut
extend the result to the kinetic case with a force field in $BV_{loc}$. 
We will often use this result in this article (see theorem 1). In the
case of two dimensionnal vector-field, we also refer to the work of
F. Bouchut and L. Desvillettes \cite{BD} in which the case of
divergence free vector-field with continuous coefficient is treated, and to my
precedent work \cite{Hau1} in which this result is extented to
vector-field with $\Ldeloc$ coefficients with a condition of
regularity on the direction of the vector-field, and to the one
dimensionnal kinetic case with a force in $\Lunloc$. The most recent work \cite{Amb} in 
this domain is from L. Ambrosio and extend the existence and uniqueness result to $BV_loc$
vector field.

Let us now define precisely what we mean by a solution of (\ref{liou_nc})-(\ref{init_nc}). 
\begin{defn}
Given an initial condition in $\Linf$, a solution of
(\ref{liou_nc})-(\ref{init_nc}) is a function $f \in \Linf([0,\infty)
\times \RR^{2dn})$ satisfying for all $\phi \in \Com([0,\infty)
\times \RR^{2dn} - I)$
\begin{multline} \label{dist}
\int_{[0,\infty) \times \RR^{2dn}} f\left( {\frac{\partial \phi}{\partial
      t} + \sum_{i=1}^{n} v_i \cdot \nabla_{x_i} \phi - \sum_{i \neq
      j} \nabla V(x_i -x_j) \cdot \nabla_{v_i} \phi} \right) \\ =
-\int_{\RR^{2dn}} f^0 \phi(0,\cdot)
\end{multline}
where $I=\{(x_1,\dots,x_n,v_1,\dots,v_n)|\exists i \neq j , \quad
x_i=x_j\}$, the set of all configurations in which at least two
particles are at the same place.

We will also use the notion of solution on the whole space. By this
we mean a function $f \in \Linf([0,\infty)
\times \RR^{2dn})$ satisfying (\ref{dist}) for all $\phi \in \Com([0,\infty)
\times \RR^{2dn})$.
\end{defn}

Remark that usually, the definition of solutions is the second
one. Indeed, as we shall see later, this two definitions are
equivalent if $\nabla V \in \Lun$ near the origin. But we shall also
deal with potentials that do not satisfy this condition, and then
the quantities in (\ref{dist}) are not defined for any $\phi \in
\Com$. This is why we introduce this two notions of
solutions. Moreover, we want to find solutions for every initial
conditions in $\Lunloc$, but with this assumption, the products in
(\ref{dist}) are not necessary well defined. Thus, we introduced as in
the work of P.L. Lions and R. DiPerna the notion of renormalized solution 
defined below.

\begin{defn}
We shall say that a mesurable function $f$ is a renormalized solution
(resp. a renormalized solution on the whole space) if $\beta(f)$ is a
solution (resp. a solution on the whole space) of
(\ref{liou_nc}) with initial conditions $\beta(f^0)$, for all $\beta \in
\CCC^1_b(\RR)$, the set of differentiable functions from $\RR$ into
$\RR$ with a bounded continuous derivative.

\end{defn}

In our proof, we will often use the following result, proved by F. Bouchut 
in \cite{Bou},
\begin{thm} \label{bouchut}
Let $f \in \Linf$ be a solution of the following equation on 
$\Omega \times \RR^m$, where $\Omega$ is an open subset of $\RR^m$
\begin{equation} \label{kin}
\partial_t f + v \cdot \nabla_x f + F(t,x) \cdot \nabla_v f =0
\end{equation}
where the force field $F$ belongs to $BV_{loc}(\Omega)$. Then, this
solution is also a renormalized one. In other words, $\beta(f)$ is
also a solution of the equation (\ref{kin}) for every $\beta \in
\CCC^1_b(\RR)$.
\end{thm} 

This kind of result is very useful, because it implies the existence
and the 
uniqueness of the solution of the transport equation and of a flow
solution of the associated ODE (in a sense which will be defined later
on). Here, we shall extend the result of F. Bouchut to vector-fields
with one singularity at the origin.

\section{Existence and uniqueness of the solutions of the Liouville equation}
\begin{thm}
Assume that $d \geq 2$, $\nabla V \in BV_{loc}(\RR^d-0)$, $\nabla
V \in  \Lunloc$ near the origin, and that there exists a positive
constant $C$ such that $V(x) \geq -C(1+|x|^2)$ a.e.. Then, for every 
initial condition in $\Lunloc$, there exists a unique renormalized 
solution to (\ref{liou_nc})-(\ref{init_nc}). 
\end{thm}

Remark that with our assumptions, the potentials is bounded by below. 
So, we do not deal with the case of the attractive coulombian potential, 
by instance.

\begin{proof}
First, we will prove that in this case, a bounded solution is always a
solution on the whole space.
 
{\bf Step 1.} Equivalence between the two notions of solutions. \\
Let $f \in \Linf(\RR \times \RR^{2dn})$ be a solution of
(\ref{liou_nc}). We want to prove that it is also a solution on the
whole space. In order to show this fact, we choose a $\phi \in
\Com(\RR^d)$ such that $(1-\phi)$ has his support in $B(0,1)$, the 
open ball of radius one centered at the origin, and that $\int 
(1-\phi) =1$, and we defined, for every $\epsilon > 0$, $\phieps 
= \phi(\cdot / \epsilon)$. Moreover, we denote
$$ 
\Phi(x_1,\dots,x_n)= \prod \phi_{\epsilon_{i,j}}(x_i - x_j) 
$$ 
where the product run over all the set of two indices $\{i,j\}$ except the set 
$\{1,2\}$, and where $\epsilon_{i,j}$  depends of $\{i,j\}$. We 
also choose an $\epsilon_{1,2}$ that we will denote by $\mu$ for 
simplification in the following. 

Next, we choose a test function $\Psi \in \Com(\RR \times \RR^{2dn})$. 
Since $\Psi \Phi \phi_{\mu}(x_1-x_2)$ has a compact support in $\RR \times 
(\RR^{2dn}-I)$, we can  write, 
\begin{multline} \label{one}
\int_{\RR \times \RR^{2dn}} f \phi_{\mu}(x_1-x_2)
\left({\frac{\partial \Psi}{\partial t} \Phi + \sum_{i=1}^n v_i \cdot
    \nabla_{x_i} (\Psi \Phi) + 
 \sum_{i \neq j} \Phi \nabla
    V(x_i-x_j) \cdot \nabla_{v_i} \Psi }\right)\\
+ \int_{\RR \times \RR^{2dn}} f \Phi \Psi \frac{1}{\mu}
\phi(\frac{x_1-x_2}{\mu}) \cdot (v_1-v_2) = 0
\end{multline}
When $\mu$ goes to $0$, the first integral goes to 
$$
\int_{\RR \times \RR^{2dn}} f
\left({\frac{\partial \Psi}{\partial t} \Phi + \sum_{i=1}^n v_i \cdot
    \nabla_{x_i} (\Psi \Phi) + \sum_{i \neq j} \Phi \nabla
    V(x_i-x_j) \cdot \nabla_{v_i} \Psi }\right) 
$$
For the second integral, we may write
\begin{eqnarray*}
\left| {\int_{\RR \times \RR^{2dn}} f \Phi \Psi \frac{1}{\mu}
\phi(\frac{x_1-x_2}{\mu}) \cdot (v_1-v_2) }\right| & \leq &
\frac{M}{\mu} \int_{|x_1-x_2| \leq \mu} \Psi \\
 & \leq & M' \mu^{d-1}
\end{eqnarray*}
and since $d \geq 2$, the second integral goes to zero.

Then, we have that
\begin{equation} \label{phi}
\int_{\RR \times \RR^{2dn}} f\left( {\frac{\partial \Psi}{\partial
      t}\Phi + \sum_{i=1}^{n} v_i \cdot \nabla_{x_i} (\Psi \Phi) -
    \sum_{i \neq j} \Phi \nabla V(x_i -x_j) \cdot \nabla_{v_i}
    \phi} \right) = 0
\end{equation}
Next, we can write $\Phi=\Phi' \phi_{\epsilon{1,3}}$. It is possible only 
if $n \geq 3$, but in the case $n=2$, $\Phi=1$ and we have already prove 
what we want. We make the same argument. Let as above 
$\epsilon_{1,3}$ going to zero and obtain (\ref{phi}) with $\Phi$ replaced 
by $\Phi'$. At this point, we can go on and do this with all the couple 
$(i,j)$, with $i \neq j$. At the end, we can delete $\Phi$ 
in the equality (\ref{one}). We obtain the equation (\ref{dist}).
Then, $f$ is a solution on the whole space.

{\bf Step 2.} Every $\Linf$-solution is a renormalized solution on the
whole space. \\
  Let $f \in \Linf$ be a solution of (\ref{liou_nc}). We choose a $\beta
\in \CCC^1_b(\RR)$. By using the theorem 1 and because the notion of
renormalisation is local, we obtain that $\beta(f)$ is a solution of
(\ref{liou_nc}). But by the step one, we know that $\beta(f)$ is a
solution on the whole space. Since this is true for every $\beta \in
\CCC^1_b(\RR)$, $f$ is a renormalized solution on the whole space.

{\bf Step 3.} Uniqueness for solution in $\Linf([0,+\infty) \times \RR^{2dn})$. \\
 We choose two solutions $f$ and $g \in \Linf([0,+\infty) \times
 \RR^{2dn})$  of (\ref{liou_nc}) with the same initial condition, and a
 $\beta \in \CCC^1_b(\RR,\RR)$, non-negative, with $\beta(0)=0$. By
 step 2 and the linearity of the equation, $h=\beta(f-g)$ is also a
 solution on the whole space of (\ref{liou_nc}) with vanishing initial conditions. 

Next, we choose a function $\psi \in \Com(\RR)$ such that $\psi \equiv 1$ on 
$(-\infty,1)$ and $\psi \equiv 0$ on $(2,+\infty)$. We also define on 
$\RR^{2dn}$ the energy E of a configuration which is given by
$$ 
E(x_1,\dots,x_n,v_1,\dots,v_n)= \frac{1}{2}\sum_{i\neq j} V(x_i-x_j) + 
\sum_{i=1}^n \frac{|v_i|^2}{2} 
$$
Remark that the assumption $V(x) \geq -C(1+|x|^2)$ implies that there
exists another constant $C >0$ such that if $E \leq R^2$ and all the
$|x_i| \leq R$ for all $i$, then $|v_i| \leq C(1+R)$ for all $i$. Roughly, if our 
particles are initially in a bounded
region, their speeds will remain bounded on every compact interval of
time. We will use this fact to prove the uniqueness.
For every $T > 0$ and $R \geq 0$, we define
$\phi_{R,T}=\psi(\sqrt{1+\sum |x_i|^2} -
(R+1)e^{C'(T-t)}-2)\psi(E/R^2)$, with $C'=nC$. $\partial_t \phi_{R,T} \in \Linfloc$,
$\nabla_{x_i} \phi_{R,T} \in \Lun$ and $\nabla_{v_i}\phi_{R,T} \in
\Linfloc$, so we may multiply the distribution $h$ by
the function $\phi_{R,T}$. We compute 
\begin{equation} \label{deriv1}
\frac{\partial (h \phi_{R,T})}{\partial t} = \phi_{R,T} \frac{\partial h}{\partial t}+
\frac{\partial \phi_{R,T}}{\partial t} h 
\end{equation}
and we obtain
\begin{multline} \label{deriv2}
\frac{\partial (h \phi_{R,T})}{\partial t} =  -\phi_{R,T} \left({\sum_{i=1}^{n} v_i 
\cdot \nabla_{x_i} h - \sum_{i \neq j} \nabla V(x_i -x_j) \cdot \nabla_{v_i} h}\right) \\
+ \frac{\partial \phi_{R,T}}{\partial t} h
\end{multline}
 Then, if we integrate this equation with respect to $x$,$v$ and use integration 
by parts, we obtain
\begin{multline} \label{argh}
\frac{\partial}{\partial t} \left({\int_{\RR^{2dn}}
h(t,\cdot)\phi_{R,T}}\right) = \\
\int_{\RR^{2dn}} h(t,\cdot) \psi'(\sqrt{1+\sum |x_i|^2} -(R+1)e^{C'(T-t)}-2)
\psi(E/R^2) \times \dots  \\
 (\sum v_i \cdot B_i - C(R+3)e^{C(T-t)})
\end{multline}
where the term
$|B_i|=|\partial_i(\sqrt{(1+\sum|x_i|^2})|=|x_i|/\sqrt{(1+\sum|x_i|^2}$
is bounded by $1$. It is useful there to
use the energy in the test function because many terms vanish when we
perform the computation, since $E$ is invariant by the flow. Now, when 
$\Phi_{R,T}$ do not vanish, it means that $E \leq R^2$ and $|x_i| \leq C R$ for all $i$.
Then, we deduce that we have
\begin{equation}
\frac{\partial}{\partial t} \left({\int_{\RR^{2dn}} h(t,\cdot)\phi_{R,T}}\right) \leq 0
\end{equation}
because when $\Phi_{R,T}$ do not vanish, $E \leq R^2$ and $|x_i| \leq R$
for all $i$. And in this conditions we have  $\sum v_i \cdot B_i -
C(R+3)e^{C(T-t)}\leq 0$ and $\psi'$ is nonpositive. Since $h$
vanishes at $t=0$, this means that
$$ \int_{\RR^{2dn}} h(T,\cdot)\phi_{R,T} =0 $$
Since this is true for every $R$ and every $T$, and since $h$ is
nonnegative, we obtain that $h$ vanishes almost everywhere on $[0,\infty)
\times \RR^{2dn}$. This is true for every $\beta \in \CCC^1_b(\RR)$ 
satisfying $\beta(0)=0$. Therefore, $f=g$ a.e..

{\bf Step 4.} Existence and uniqueness for initial conditions in
$\Lunloc$. \\
First, we remark that if $f^0 \in \Linf$, it is easy to obtain a
solution on the whole space of (\ref{liou_nc})-(\ref{init_nc}) by
regularisation of the force field and the use of weak limits. In
addition, in view of the result obtained in step 3, we obtain that,
if  $f^0 \in \Linf$, there exists a unique solution of
(\ref{liou_nc})-(\ref{init_nc}) which is also a renormalized solution on
the whole space.

Next, let $f^0 \in \Lunloc$. For $m \in \NN$, we define $\beta_m(x)=(x
\wedge m)\vee  -m$ , and we remark that $\beta_m \circ \beta_p=
\beta_p$, if $p \geq m$. For all $m \in \NN$, there exists a unique
solution of  (\ref{liou_nc})-(\ref{init_nc}) corresponding to the initial condition
$\beta_m(f^0)$. We denote it by $f_m$. For every $p \in \NN$, $f_p$ is a
renormalized solution on the whole space, then $\beta_m(f_p)$ is a
solution with intitial conditions
$\beta_m(\beta_p(f^0))=\beta_m(f^0)$. Of course, $\beta_m$ do not
belongs to $\CCC^1_b(\RR)$ but it can be shown that the
renormalisation property is still true for Lipschitz function by
regularisation of those functions (see \cite{Bou}). Then, by the uniqueness of the solution
of (\ref{liou_nc})-(\ref{init_nc}) when the initial condition belongs to $\Linf$,
we obtain that $\beta_m(f_p)=f_m$, for all $p \geq m$. This allows us
to define almost everywhere 
$$f=lim_{m \rightarrow \infty} f_m$$ 
This measurable function $f$ satisfies $\beta_m(f)=f_m$ for all $m \in
\NN$. And $f$
is a renormalized solution corresponding to the initial condition $f^0$ because for
every $\beta \in \CCC^1_b(\RR)$, $\beta \circ \beta_m(x)$ goes to
$\beta(x)$ a.e. in $x$ when $m$ goes to $+\infty$. Then, the solution
$\beta \circ \beta_m(f)$ of (\ref{liou_nc})-(\ref{init_nc}) with the initial
condition $\beta \circ \beta_m(f^0)$ goes a.e. to $\beta(f)$ which is
still a solution of (\ref{liou_nc})-(\ref{init_nc}) with the initial
condition $\beta(f^0)$, because this linear equation is always
satisfied by a weak limit of solutions. This shows the existence of the solution. 
For the uniqueness, if there existed two solutions $f,g$ for the same
initial conditions $f^0$, there would be a $m \in \NN$ such that
$\beta_m(f) \neq \beta_m(g)$, and $\beta_m(f)$,$\beta_m(g)$ would be
two distinct solutions in $\Linf$ with the same initial
conditions. This would contradict the uniqueness of solutions already
proved in that case. 

Finally, we remark that we can say something about the
integrability of the solution $f$. Since the speed of propagation is finite on the sets of
bounded energy, $f \tr \in \Linfloc(\RR, \Lunloc)$, for all $m \in
\RR$, where  
$\chi_{E<m}$ denote the characteristic function of the set of all the 
configurations with an energy less than $m$. 
\end{proof}

\begin{thm}
Assume now that $\nabla V \in BV_{loc}(\RR^d-0)$, that $V$ is bounded
on all compact sets of $\RR^d -0$, that $V$
satisfies $V(x) \geq C(1+|x|^2)$ a.e. and that $V$ goes to
$+\infty$ when $|x|$ goes to $0$. Then, there exists a unique
renormalized solution of (\ref{liou_nc})-(\ref{init_nc}). 
\end{thm}
\begin{proof}
  The proof will follow the same sketch that the one of the theorem
  2, but the difficulties are at others places. First, the existence
  of solution by regularisation is not so obvious here, because we
  cannot work on the whole space.

{\bf Step 1.} Existence of solution with initial condition in
$\Linf$. \\
  We choose a smooth $f^0 \in \Linf$. We shall show the existence of a
  solution with this initial condition by regularisation. We choose a
  regularisation kernel $\rho \in \Com(\RR^d)$, such that $Supp(\rho)
  \subset B_1$, and that $\int \rho =1$. We also choose a smooth
  function $\alpha$ from $\RR^n$ into $\RR$ satisfying $\alpha(x) \leq
  min(1,|x|/2)$ for all $x$. We denote $\rho_{\epsilon}= \rho(\cdot
  /{\epsilon})$ and define for all integer $n \geq 1$
$$ V_n(x)= \int_{\RR^d} V(y) \rho_{2^{-n}\alpha(x)}(x-y)\,dy $$
It is a sort of convolution, in which the radius of the ball on which
we average $V$ depends on $x$ so that $0$ is never in that ball. Hence
$V_n$ is well defined in $\RR^d - 0$, belongs to $\Com(\RR^d -
0)$ and satisfies also $V_n(x) \rightarrow +\infty$ when $|x|
\rightarrow 0$. Moreover, $\nabla V_n \rightarrow \nabla V$ in $BV_{loc}(\RR^d
-0)$ when $n \rightarrow \infty$.

  Then, if $Y=(X,V)$ is such that $X \notin I$, there exists a unique maximal
solution to the ODE with value $(X,V)$ at time $t=0$. Because of the 
conservation of the energy, it cannot reaches $I$ and because of property of 
$V_n$, it cannot go to infinity in a finite time. Then, this maximal solution 
is defined for every time. This allows us to define a
smooth flow $Y_n(t,\cdot)$ in $\RR^{2dn}-I$. And $f_n=f^0(Y_n)$
satisfies the Liouville equation in the classical sense on
$\RR^{2dn} -I$. Then, $f_n$ also satisfies (\ref{dist}), for all test
functions $\phi \in \Com(\RR^{2dn}-I)$. 

Moreover, the sequence $(f_n)$ is bounded by $\|f^0\|_{\infty}$ in
$\Linf$, then, up to an extraction, we can assume that $f_n
\rightarrow f$ weakly in $\Linf-w \ast$. And we can pass to the limit in
(\ref{dist}) and obtain that $f$ is a solution of
(\ref{liou_nc})-(\ref{init_nc}). For non smooth initial condition  $f^0 \in
  \Linf$, we obtain the existence of the solution by regularistion of
  $f^0$ and by taking weak limit.

{\bf Step 2.} Uniqueness of solution for initial conditions in $\Linf$. \\
Here we choose an $h$ solution of (\ref{liou_nc}) with vanishing intial
conditions. Now, with our assumption that $V(x) \rightarrow +\infty$
when $|x| \rightarrow 0$, the support of function is included in
$\RR^{2dn}-I$. As in the proof of the theorem 2, we may write the
equations (\ref{deriv1}) and (\ref{deriv2}), and not only for test functions vanishing on
$I$, but for every smooth functions with compact support in
$\RR^{2dn}$, because of the property of the support of $\phi_{R,T}$. 
So, we obtain (\ref{argh}), and then that, 
$$ \frac{\partial}{\partial t} \left({\int_{\RR^{2dn}}
    h(t,\cdot)\phi_{R,T}}\right) \leq 0 $$
And this implies the uniqueness of the solution. 

The last step about existence and uniqueness of the solution with
initial conditions in $\Linfloc$ is the same that in the theorem 2.
\end{proof}

\section{Resolution of the ordinary differential equation}

  We are now looking for a solution to the ODE associated to the transport
equation. Since the vector-field used in this ODE is not defined
everywhere, we cannot solve this ODE for every initial condition. Then, we 
will solve it globally with a flow, namely a application $Y$ from $\RR 
\times 
\RR^{2dn}$ to $\RR^{2dn}$ such that $Y(t,Y_0)$ is the position in the 
phase space at time $t$ when we start from $Y_0$ at time $0$. Of course, 
this flow will be defined only almost everywhere. Here we use the notation 
$Y=(X,V)=(x_1,\dots,x_n,v_1,\dots,v_n)$, where $X$,$V \in \RR^{dn}$ and 
$x_i$,$v_i \in \RR^d$ for all $i$. This flow shall solve the 
following system
\begin{equation} \label{flow}
\left\{ \begin{array}{l} 
\dot{x_i}(t,Y) = v_i \\
\dot{v_i}(t,Y)= - \sum_{j\neq i} \nabla V(x_i - x_j) \\
Y(0,X,V)=(X,V) \end{array} \right.
\end{equation}
  If we denote by $B$ the vector-field defined below on $\RR^{2dn}$,
  we may rewrite the two first equations 
$$\dot{Y}=B(Y)$$
$$ \text{where} \quad B(Y)= (v_1,\dots,v_n,- 
\sum_{j\neq 1}
\nabla V(x_1 - x_j),\dots,- \sum_{j\neq n} \nabla V(x_n - x_j)) $$
  In our situation of a vector field with low regularity, we have to say 
more precisely what 
we will mean by a flow, and in which sense we look at the equation 
(\ref{flow}). This is the aim of the following definition, in which
$\chi_{E<m}$ denote the characteristic function of the set of all the 
configurations with energy less than $m$. 

\begin{defn}
A flow defined almost everywhere (a.e. flow) solution of the ODE
(\ref{ODEnpart}) is a function $Y$ from $\RR \times \RR^{2dn}$ to
$\RR^{2dn}$ such that
\begin{itemize}
\item[i.] $Y \tr \in \CCC(\RR, \Lunloc)^{2dn} \cap
  \Linfloc(\RR^{2dn+1}), \; \forall m \in \RR$
\item[ii.] $\int \phi(Y(t,X,V))\,dXdV = \int \phi(X,V)\,dXdV, \quad 
\forall \phi \in \Com, \quad \forall t \in \RR$
\item[iii.] $Y(t+s,Y')=Y(t,Y(s,Y'))$ a.e. in $Y', \quad \forall s,t \in 
\RR$
\item[iv.] $E \circ Y(t,Y') = E(Y')$ (the energy is preserved by the 
flow).
\item[v.] $\dot{Y} \tr =B(Y) \tr$ is satisfied 
in the sense of the distributions for all $m \in \RR$, and 
$Y(0,X,V)=(X,V)$ a.e. on $\RR^{2dn}$.
\end{itemize}
\end{defn}

{\bf Remark.} We use the truncation $\tr$ because in the 
region where $E$ is large, the particles may go to infinity very quickly 
and we cannot expect $Y$ to be integrable. It has the avantage to allow us 
to give a sense to the EDO without using renormalization, like in 
\cite{DPL}. But, this definition is not completely satisfactory because we 
like part iv. to be a consequence of the others points, but I do not 
know how to do this. 

  Using the results of the first section, we will prove the existence and 
the 
uniqueness of an a.e. flow in the two case seen above. For this, we use 
the method introduced by  R. DiPerna and P.L. Lions in \cite{DPL}. Indeed, we 
just adapt the argument introduced in \cite{Lio} for periodic vector-fields.

\begin{thm}
Under the two kind of assumptions made in the section 1, 
there exists a unique a.e. flow solution of (\ref{flow}).
\end{thm}
\begin{proof}
  We first remark, that in the case when $V$ is smooth, a flow solution of 
an ODE is 
also a solution of the transport equation (more precisely each component 
$Y_i$ is the unique solution corresponding to the initial condition 
$f^0(Y)=Y_i$). Here, we will use this remark, and the fact 
that we know how to solve the transport equation. 
We thus denote by $Y$ the solution of the transport equation (\ref{liou_nc})  
for the initial condition $f^0(Y)=Y$. We will prove that this 
defined an a.e. flow solution of (\ref{flow}). 

In the first section, we have shown that $Y$ is a renormalized solution of
(\ref{liou_nc}). Let us recall that it means that $\beta(f)$ is a solution of 
the Liouville equation, for every $\beta \in \CCC^1$. But here the initial 
condition belongs to $\Linfloc$. And we point out that in both cases of 
section 1, we have proved that the speed of propagation is finite on the
set where the energy is bounded. Then, assume that $f$ is a solution with
an initial condition given by $f^0 \in \Linfloc$. We choose a smooth
function $\psi \in
\Com(\RR)$. We may prove adapting the argument made in section 1,
that for every $R$ and $T \in \RR$, there exists a constant $R'>0$ such
that
$$ \int_{|x_i|,|v_i|
  \leq R} \beta(f(t,Y))^n \psi^n(E)\,dY \leq \int_{|x_i|,|v_i|
  \leq R'} \beta(f^0)^n \psi^n(E)\,dY $$
for all $\beta \in \CCC^1_b$, and all $n \in \NN$.
Since this is true for all $n$ and all $\beta$ we obtain that for every $m
\in \RR$, $f \tr \in \Linfloc(\RR^{2dn+1})$. This implies that $f \tr$
is a solution (not only a renormalized solution) of (\ref{liou_nc}).

Next, we shall show that we can extend the renormalisation property to
functions of several variables. More precisely, if $G \in \CCC(\RR^{k})$
and $f_1,\dots,f_k$ are solution in $\Linfloc$, then
$G(f_1,\dots,f_k)$ is also a solution of the same equation, with
initial conditions $G(f_1^0,\dots,f_k^0)$. Let us show the proof for $k=2$
for example.

Thus, take $f$ and $g \in \Linfloc$
two solutions of the transport equation (\ref{liou_nc}). Next, $(f+g)$,
$(f-g)$ are also solutions by linearity, and so are $(f+g)^2$,$(f-g)^2$,
and finally $fg=(1/4)[(f+g)^2-(f-g)^2]$. Doing this again, we can show
that P(f,g) is also a solution for all $P$ polynomial in two variables. 
And using the density of the polynomials, we finally obtain that G(f,g) is 
a solution for every continious $G$.

Then, for all $f^0 \in \Com(\RR^{2dn})$, $f^0(Y(t,Y'))\chi_{E<m}$ is the 
solution with initial conditions $f^0\chi_{E<m}$. Letting $m$ going to 
$\infty$, we obtain that $f^0(Y(t,Y'))$ is the solution with initial 
conditions $f^0$. And this is true for every $f^0 \in \Linf$ by 
approximation.
  Next, since the Liouville equation preserves the total mass, we obtain 
that $\int f(Y(t,Y'))\,dY'=\int f(Y')\,dY'$, for every smooth $f$. This
implies the part ii. of the definition of an a.e. flow (the conservation 
of the Lebesgue 
measure). 
  
  For the group property $Y(s+t,Y')=Y(t,Y(s,Y'))$ a.e. in $Y'$, we choose a fixed $t$ 
and a sequence of smooth function going to $Y(t,\cdot)$ in $\Lunloc$. 
Because of the part ii. of the definition, $f(Y(s,\cdot))$ goes to
$Y(t,Y(s,\cdot))$. But, since $f$ goes to $Y(t,\cdot)$ in $\Lunloc$,
$f(t,\cdot)\tr$ goes in $\Lunloc$ to the solution of (\ref{liou_nc}) with
initial conditions $Y(t,\cdot)\tr$ at time $s$. This is
$Y(s+t,\cdot)\tr$. And the group properties follows.

To show that the energy $E$ is invariant by the flow (part iv.),
remark that $E \circ Y$ and $E$ are two solutions of (\ref{liou_nc}) with
the same initial conditions $E$. Then, they are equal.

  In order to show the part v., we choose $\phi \in \Com(\RR^{2dn})$ and 
$\psi \in \Com(\RR)$. We will use the function $\phi\psi$ as test
function. It is sufficient to use only this type of functions to show
that $f$ satisfy the equation, because linear combinations of such
functions are dense in the space $\CCC^1_o(\RR\times \RR^{2dn})$. We
compute for all $i \leq 2dn$, where the index $i$ denote the $i-th$
component of vector in $\RR^{2dn}$ 
\begin{gather*}
\begin{split}
 \int_{\RR \times \RR^{2dn}} Y_i(t,Y') & \tr \phi(Y')\frac{\partial
  \psi}{\partial t} (t)\,dYdt \\  
 & =  \int_{\RR \times \RR^{2dn}}
  Y_i(-t,Y')\tr \phi(Y')\frac{\partial \psi}{\partial t}(-t)\,dY'dt \\
 & =  \int_{\RR \times \RR^{2dn}}
  \phi(Y(t,Y')) \tr Y'_i\frac{\partial \psi}{\partial t}(-t)\,dY'dt
\end{split}
\end{gather*}
To obtain the second equation from the first, we use the change of
variable $Y(t,\cdot)$. And we remark that $\tr \circ Y = \tr$, since the 
energy is invariant by the flow. 

Moreover, we know that $\phi(Y(t,Y'))$ is the solution of the
transport equation (\ref{liou_nc}) with initial conditions $\phi$. We use
this to write 
\begin{gather*}
\begin{split}
 \int_{\RR \times \RR^{2dn}} Y_i(t,Y') & \tr \phi(Y')\frac{\partial
\psi}{\partial t}(t)\,dY'dt \\ 
 & = - \int_{\RR \times
\RR^{2dn}}\phi(Y(t,Y')) \tr  B_i(Y') \psi(-t)\,dY'dt \\
 & = -\int_{\RR \times \RR^{2dn}} B_i(Y(-t,Y')) \tr 
\phi(Y')\psi(-t)\,dY'dt \\
 & = -\int_{\RR \times \RR^{2dn}} B_i(Y(t,Y')) \tr 
\phi(Y')\psi(t)\,dY'dt \\
\end{split}
\end{gather*}
And this shows the part iv.. Therefore, the existence of such a solution 
is proven. Remark that in the second case we can
delete $\tr$ if we only use test functions whose support does not
contain $0$.

For the uniqueness of the a.e. flow, we will show that the five
properties satisfied by this flow implies that all his components are
solutions of the Liouville equation. This is sufficient to show the
uniqueness of an a.e. flow because we already know the uniqueness of
the solution of the Liouville equation.

We choose $\phi \in \Com(\RR^{2dn})$ and $\psi \in \Com(\RR)$ and use
$\phi \psi$ as test function. We have for all $i \leq 2dn$ that
\begin{gather*}
\begin{split}
 \int_{\RR \times \RR^{2dn}} Y_i(t,Y')\phi(Y') & \tr\frac{\partial
  \psi}{\partial t}(t)\,dY'dt \\ 
 & = \int_{\RR \times \RR^{2dn}} Y_i'
  \phi(Y(-t,Y'))\tr\frac{\partial \psi}{\partial t}(t)\,dY'dt  \\
 & =  -\int_{\RR \times \RR^{2dn}} Y'_i \tr\frac{\partial}{\partial t}\big(
  \phi(Y(-t,Y')) \big) \psi(t)\,dY'dt
\end{split} 
\end{gather*}
 In the first equality we use the change of variable $Y'=Y(t,Y')$, and the
second one is deduced by an integration by parts. Remark that we use 
the preservation of the energy by the flow in every change of variable. 
But we can show that in a $\Lunloc$-sense,
$$  \frac{\partial}{\partial t}\big(\phi(Y(-t,Y'))\tr \big) =
-\nabla\phi(Y(-t,Y')) \cdot B(Y(-t,Y'))\tr$$
This, because for $t$ fixed, $\frac{Y(t+h,Y')-Y(t,Y')}{h}\tr \rightarrow
B(Y(t,Y'))\tr$ in $\Lunloc(\RR^{2dn})$ when $h \rightarrow 0$.
Let us show this fact. Indeed, if we look at the five properties satistied
by an a.e. flow, we can show that
$$ Y(t,Y')\tr=Y' \tr + \int_0^t B(s,Y(s,Y')\,ds  \quad \text{a.e. in Y',
 } \forall t \in \RR.$$
It remains to show that $\tr B(Y) \in \CCC(\RR, \Lunloc)$ to obtain the
result. For this, if $B$ is replaced by a smooth and bounded
$B_{\epsilon}$, this is true, because $Y\tr \in \CCC(\RR,\Lunloc)$. And
this is still true for $B$ because $Y$ preserves the Lebesgue measure and
because the energy is preserved by the flow.

 Then, we obtain if we use the change of variables backwards
\begin{multline}
\int_{\RR \times \RR^{2dn}} Y_i(t,Y')\tr \phi(Y')\frac{\partial
\psi}{\partial t}(t)\,dY'dt   \\ =   \int_{\RR \times \RR^{2dn}}
Y_i(t,Y') \tr  B(x) \cdot \nabla \phi(x) \psi(t)\,dY'dt
\end{multline}
And $Y_i(t,Y')$ satisfies (\ref{liou_nc}) and the proof is complete.

\end{proof}

\def\cprime{$'$} \def\cprime{$'$}

\end{document}